\documentclass{amsart}
\usepackage{amsmath}
\usepackage{graphicx}
\usepackage{hyperref}
\usepackage{amsthm}
\usepackage{amssymb}
\usepackage{todo}
\usepackage{graphicx}
\usepackage{geometry}
 \usepackage{color}
\usepackage[pagewise]{lineno}

\newtheorem{thm}{Theorem}[section]
\newtheorem{cor}[thm]{Corollary}
\newtheorem{lem}[thm]{Lemma}
\newtheorem{prop}[thm]{Proposition}
\theoremstyle{definition}
\newtheorem{defn}[thm]{Definition}

\newtheorem{rem}[thm]{Remark}
\numberwithin{equation}{section}

\begin{document}

\author{Xiaoyan Su $^1$, Shiliang Zhao$^2$, Miao Li$^3$}

\title{Dispersive estimates for time and space fractional  Schr\"odinger equations}
\maketitle
\begin{abstract}
In this paper, we consider the Cauchy problem for the fractional Schr\"odinger equation $i D_t^\alpha u + (-\Delta)^{\frac{\beta}{2}} u =0$ with $0<\alpha<1$, $\beta>0$. We establish the dispersive estimates for the solutions. In particular, we prove that the decay rates are sharp.

\end{abstract}




\section{Introduction}
Fractional differential equations { have }become a very important subject in the modern mathematical research because of its various applications in physics, biology, probability and so on. We refer the {  readers } to  \cite{17}\cite{14}\cite{19} for a systematic introduction about fractional calculus.  In probability, it has a close connection  to the time-changed stochastic process, for more details, we refer the { readers} to \cite{15}.  In physics, fractional calculus are introduced to study the quantum phenomena, see \cite{8}\cite{12}\cite{13}\cite{21}.

It is well known that the Schr\"odinger equation is one of the most important { equations} in the quantum mechanics, and the well-posedness of this equation has attracted the attention of thousands of researches, see \cite{4}\cite{5}\cite{20} for more information.
And its fractional generalisation can be divided into mainly three fields: (1) space fractional Schr\"odinger equation (SFSE), see\cite{8,12,22}; (2) time fractional Schr\"odinger equation (TFSE), see\cite{1,16,23}; (3) and both time and space fractional Schr\"odinger equation (TSFSE), see \cite{21,24,25}.

Mathematical literatures about space fractional Schr\"odinger equation include \cite{7,10,11} and references therein. From the physics viewpoint, (SFSE) can be derived by using the Feynman path integral techniques by replacing the Brownian quantum paths  with L\'evy stable paths, and the Markovian character are still remained, see \cite{12} .

As for the time fractionalisation of the Schr\"odinger equation, the case is more controversial, mainly about if we should fractionalize the constant $i$. In \cite{16}, the author use Wick rotation to raise a fractional power of $i$, which turns out to be the clasical Schr\"odinger equations with a time  dependent Hamiltonian. Whereas in the literature of \cite{1} the authors derived the (TFSE) using the Feynman paths method, which doesn't change the constant $i$.

In this paper, we study the fractional Schr\"odinger equation in both time and space, which is a generalisation of the equation in  \cite{1}:

\begin{equation}\label{fse}
	\left\{\begin{split}
		iD_t^\alpha u(t,x) &+(-\Delta )^{\frac{\beta}{2} }u(t,x)=0, \\
		u(0,x)&= \varphi (x),
	\end{split}
	\right.
\end{equation}
where $0<\alpha<1, \beta>0 $, $D_t^\alpha$ denotes the Caputo fractional derivative and $(-\Delta)^{\frac{\beta}{2}}$ denotes the fractional Laplacian.

The aim  of this paper is to study the dispersive estimates of the equation \eqref{fse}. When $\alpha=1, \beta=2$, \eqref{fse} becomes the classical Schr\"odinger equation. As is well known, the following estimates hold:
\[\|u(t)\|_{L^{p'}(\mathbb R^n)}\leq t^{-n(\frac{1}{p}- \frac{1}{2} )}\|u\|_{L^{p}(\mathbb R^n)}\quad \text{for}\quad 1\le p\leq 2, t>0.\]
The above decay estimates are very important in the nonlinear Schr\"odinger equation theory. See for example \cite{5} and references therein. When $\alpha=1, \beta>0 \text{ and}\quad \beta\neq1$, the dispersive estimates have been established in \cite{6}, and { they have also shown that the decay rates are sharp.}

In the first place, we get a $L^\infty$ estimates of the fundamental solutions of \eqref{fse}. In fact, the solution of \eqref{fse} can be expressed as a Fourier multiplier,
$$  u(x,t)=[E_\alpha (-it^\alpha|\xi|^{\beta} ) \hat\varphi(\xi)]^{\vee}(x)\quad t>0,   $$
where $\varphi\in \mathcal S(\mathbb R^n) $ and $E_\alpha(z)$ is the Mittag-Leffler function.
Since we have $|E_\alpha(z)|\leq C$ for $0<\alpha<1, \frac{\pi}{2}\alpha<|\arg z|<\pi $ (see \cite{18} ),
it follows that $\|u(t)\|_{L^2(\mathbb R^n )}\leq C\|\varphi \|_{L^2(\mathbb R^n )}$.

Set $T_t\varphi(x)=[E_\alpha (-it^\alpha|\xi|^{\beta} ) \hat\varphi(\xi)]^{\vee}(x)= K_t \ast \varphi (x) $, where $K_t(x)=[E_\alpha (-it^\alpha|\xi|^{\beta} )]^{\vee}(x)$  denotes the distributional kernel of $T_t$. Then we have the following result:
\begin{prop}
	For any fixed $t>0, K_t(\cdot)\in L^\infty(\mathbb R^n)$ iff $ \beta>n$.
\end{prop}
Note that  when $\alpha=1, \beta>0$, similar results have been proved by \cite{7}.  {{Compared with the classical Schr\"odinger equation, where $\alpha=1$, $\beta=2$ }, and $K_t(x)=\frac{1}{(4\pi i )^{\frac{n}{2}}}e^{\frac{i|x|^2}{4t}}\in L^\infty (\mathbb R^n)$ for any fixed $t>0$, our result requires the condition that $\beta>n$ to make sure that $K_t(\cdot)\in L^\infty(\mathbb R^n)$, this phenomenon is caused by the different asymptotic behaviours of $E_\alpha (-i|\xi|^2) $ for $0<\alpha <1$  and  $\alpha=1$.
When $0<\alpha <1, E_\alpha (-i|\xi|^2)$ has quadratic decay, whereas when $\alpha=1, E_\alpha (-i|\xi|^2)=e^{-i|\xi|^2}$, which doesn't have any decay when $|\xi|\rightarrow \infty $ but oscillates. }

As a result, when $\beta>n$, we can easily get the $L^\infty $ estimates for the solutions of \eqref{fse} by Young's inequality.
 As for the case  $0<\beta\leq n$, we need to localize the frequencies.

  Let us consider a function $\phi \in C_c^\infty$ such that
\[\phi (\xi)=\begin{cases}
 	1, & \text{if}\quad |\xi|\leq 1,\\
 	0, & \text{if}\quad |\xi|\geq 2,
 \end{cases}\]
 and  define the sequence $(\psi_j)_{j\in \mathbb Z}\subset \mathcal{S}(\mathbb R^n)$ by $\psi_j(\xi)=\phi (\frac{\xi}{2^j} )-\phi (\frac{\xi}{2^{j-1}} )$. We use $\psi$ to denote $\psi_0$. It is easy to notice that $ \text{supp} \psi_j \subset \{2^{j-1}\leq |\xi|\leq 2^{j+1}\}.$
Set $P_N\varphi(x) =[\psi_j (\xi ) \hat\varphi(\xi) ]^{\vee}(x)$, where $N$ is dyadic number, i.e. $N=2^j$ for some {  integer $j$.}
The main results of this paper are as follows:
\begin{thm}\label{Lebesgue estimate}
	\begin{enumerate}
		\item For $0<\alpha<1, \beta>0$ and {  dyadic number $N$}, we have
		\[\|P_NT_t\varphi\|_{L^{\infty}(\mathbb R^n)}\lesssim \frac{N^n}{1+t^{\alpha}N^{\beta}}\|\varphi \|_{L^1(\mathbb R^n)}. \]
		\item For $0<\alpha<1, \beta>n$, we have
		\[\|T_t \varphi\|_{L^\infty(\mathbb R^n)}\lesssim t^{-\frac{n}{\beta}\alpha}\|\varphi\|_{L^1(\mathbb R^n)}. \]
	\end{enumerate}
\end{thm}
{  Notice that when $0<\beta\leq n$, for fixed $N$, the time decay is $t^{-\alpha}$, which doesn't depend on the dimension. That is different from the classical case. }	

Besides, the above decay estimates are sharp in the following sense:
\begin{prop}\label{sharpness}
	There exists $t_0, N_0$ such that for all $t>t_0, N>N_0$, we have
	\[\sup_{x\in \mathbb R^n} |K_t(x)|\gtrsim \frac{N^n}{1+t^\alpha N^\beta}. \]
	
\end{prop}
Moreover, we have the following dispersive estimates:
\begin{thm} \label{Besov estimate}

For $0<\alpha<1, 2\leq r\leq \infty$, we have
	\begin{equation}\label{eq7}	
	\|T_t\varphi \|_{L^\infty(\mathbb R^n)}\lesssim (1+t^\alpha)^{-(1-\frac{2}{r})} \bigg[ \|\varphi  \|_{\dot{B}^{\frac{n}{r'}}_{r', 1}  } +  \|\varphi  \|_{\dot{B}_{r',1}^{\frac{n}{r'}-\beta (1-\frac{2}{r} )}(\mathbb R^n)} \bigg],\end{equation} 
	 
	\begin{equation}\label{eq8}
	\|T_t\varphi \|_{L^r(\mathbb R^n)}\lesssim (1+t^\alpha)^{-(1-\frac{2}{r})}\bigg[ \|\varphi  \|_{\dot{B}^{n(1-\frac{2}{r})}_{r', 2}  }+  \|\varphi \|_{\dot{B}_{r',2}^{(n-\beta) (1-\frac{2}{r} )}(\mathbb R^n)} \bigg].\end{equation}
and
\begin{equation}\label{eq9}
 {\|T_t\varphi \|_{\dot{B}^{s}_{r, p}  }\lesssim (1+t^\alpha)^{-(1-\frac{2}{r})}\bigg[ \|\varphi  \|_{\dot{B}^{n(1-\frac{2}{r})+s}_{r', p}  }+  \|\varphi \|_{\dot{B}_{r',p}^{(n-\beta) (1-\frac{2}{r} )+s}(\mathbb R^n)} \bigg].}\end{equation}
\end{thm}
\begin{rem}
	 The Strichartz estimates for the classical Schr\"odinger equation plays an important role in the well-posedness for the semilinear equations. But due to the loss of the semigroup property for the Mittag-Leffler functions, we can not use the $T T^*$ argument to prove the Strichartz estimates for the solutions of \eqref{fse} directly.
\end{rem}

Throughout this paper, we say $f\sim g$ if and only if there exists positive constants $c, C$ such that $cf\leq g\leq C f$. We say that $f\lesssim g$
if and only if there exists a positive constant $C$ such that $f\leq C g$.

\section{Preliminaries}

In this section, we give some explanations about the notations and
  introduce some basic knowledge about the Mittag-Leffler functions and Bessel functions and Besov spaces which will be used in the later sections.
  \begin{defn}[see \cite{19} ] The fractional Riemann-Liouville integral of order $\alpha >0$ is defined as
\[J^{\alpha} f(t)= \frac{1}{\Gamma(\alpha)}\int_0^t(t-s)^{\alpha-1}f(s)\,ds,\]
where \begin{equation*}
g_{\alpha}(t)=\left\{\begin{aligned}
	&\frac{t^{\alpha-1}}{\Gamma(\alpha)}\quad &\text{when} \quad &t>0£¬ \quad \\
	&0\quad &\text{when} \quad &t\leq 0£¬ \quad \text{ʱ}¡¡
	\end{aligned}\right.
	\end{equation*}
and $\Gamma(\alpha )$ is the Gamma function. Set moreover $g_0(t):= \delta(t)$, the Dirac delta-function. The fractional Caputo derivative of order $\alpha\in (0,1)$ is defined as
\[(D_t^\alpha f)(t)=\int_0^t g_{1-\alpha}(t-s)f'(s)ds= J^{1-\alpha}f'(t). \]
  \end{defn}
  We use the notation $ \hat{f}$, $\check{f}$ to denote the Fourier  and inverse Fourier transform of a function respectively, which is given by
  \[ \hat{f}(\xi)= \frac{1}{(2\pi)^{\frac{n}{2}}}\int_{\mathbb R^n}e^{-ix\cdot \xi}f(x)dx ,\]
  \[\check{f}(x)= \frac{1}{(2\pi)^{\frac{n}{2}}}\int_{\mathbb R^n}e^{ix\cdot \xi}f(\xi)d\xi .\]
 { Then the fractional Laplacian can be defined  as}
\[ (-\Delta)^{\frac{\beta}{2}}f(x)= [|\xi|^\beta \hat{f}(\xi)]^{\vee}(x). \]
\subsection{Mittag-Leffler function}
	The Mittag-Leffler function is defined by the following power series:
	\[E_\alpha(z)=\sum_{k=0}^{\infty}\frac{z^k}{\Gamma(\alpha k+1 )}, \quad\Re \alpha>0 .\]

In this paper, we need the following asymptotic expansion of Mittag-Leffler function as $|z|\rightarrow \infty$ in various sectors of complex plan in the following proposition:
\begin{prop}[Proposition 3.6, \cite{18} ]
Let $0<\alpha<2$, and $\frac{\pi}{2}\alpha<\theta<\min\{\pi, \alpha\pi\} $, then we have the following asymptotic formula in which $k$ is an arbitrary positive integer:
\begin{align*}
 E_\alpha(z)&=\frac{1}{\alpha}\exp(z^\frac{1}{\alpha} )-\sum_{j=1}^k\frac{z^{-j}}{\Gamma(1-\alpha j )}+O(|z|^{-1-k}), \quad |z|\rightarrow\infty, \quad |\arg z | \leq \theta, \\
  E_\alpha(z)&=-\sum_{j=1}^k\frac{z^{-j}}{\Gamma(1-\alpha j )}+O(|z|^{-1-k}), \quad |z|\rightarrow\infty, \quad \theta \leq |\arg z | \leq \pi.
 \end{align*}
\end{prop}
\subsection{Bessel function}[Appendix B, \cite{9} ]
For every $\nu\in \mathbb C$ such that $\Re v>-\frac{1}{2} $, the Bessel function of the first kind and of complex order $\nu$ can be defined by the Poisson representation formula:
\[J_\nu(z)=\frac{1}{\Gamma(\frac{1}{2} )\Gamma(\nu+\frac{1}{2} )} \left(\frac{z}{2} \right)^{\nu}\int_{-1}^{1}e^{izt}(1-t^2)^{\frac{2\nu-1}{2}}dt.\]
Another expression of $J_\nu$ as a power series for an arbitrary value if $\nu\in \mathbb C$ is provided by
\[J_\nu (z)=\sum_{k=0}^{\infty}(-1)^k\frac{(\frac{z}{2})^{\nu+2k}}{\Gamma(k+1) \Gamma(k+\nu+1)}, \quad |z|<\infty, \quad |\arg z| <\pi.\]
For the asymptotic behaviour of $J_\nu(z)$ as $z\rightarrow 0 $ or $z\rightarrow \infty$, we have the following proposition.
\begin{prop}[Appendix B, \cite{9} ] Let $\Re \nu >- \frac{1}{2} $, we have
\begin{enumerate}
	\item $J_\nu(z)\sim \frac{2^{-\nu}z^\nu}{\Gamma(\nu+1)}, \quad \text{as} \quad z\rightarrow 0$.
	\item $J_{\nu}(z)=\sqrt{\frac{2}{\pi z}}\cos(z-\frac{\pi\nu }{2}-\frac{\pi}{4} )+O(|z|^{-\frac{3}{2}}), \quad \text{as}\quad |z|\rightarrow\infty, \quad -\pi+\delta<\arg z<\pi-\delta. $
\end{enumerate}
In particular, we have $J_{-\frac{1}{2}}(z)= \sqrt{\frac{2}{\pi z}}\cos(z).$	
\end{prop}
\subsection{Sobolev and Besov spaces}[Section 6.2, \cite{2} ]
%
%
%

 Now we introduce the homogenous Sobolev spaces and Besov spaces. Let $s\in \mathbb R$ and $1\leq p, q \leq \infty$,  the homogenous Sobolev space is defined by
 \[\dot{H}^{s,p}(\mathbb R^n)=\{\varphi \in \mathcal{S}'(\mathbb R^n): \|\varphi\|_{\dot{H}^{s,p}(\mathbb R^n)}<\infty\}\]
 where $\|\varphi\|_{\dot{H}^{s,p}(\mathbb R^n)}:= \|\sum_{j=-\infty}^{\infty}(|\xi|^s\psi_j \hat{\varphi})^{\vee} \|_{L^p(\mathbb R^n)} $, and the homogenous  Besov space is given by

  \[\dot{B}_{p,q}^{s}(\mathbb R^n)=\{\varphi\in \mathcal{S}'(\mathbb R^n): \|\varphi\|_{\dot{B}_{p,q}^{s}(\mathbb R^n)}<\infty\}\]
  where \begin{equation*}
	\|\varphi\|_{\dot{B}_{p,q}^s}=
 	\begin{cases}
 		(\sum_{-\infty}^{\infty}(2^{jsq}\|(\psi_j \hat {\varphi})^{\vee}\|_{L^p})^q)^{1/q},& \quad 1\leq q <\infty\\
		\sup_{j\in \mathbb Z} 2^{js} \|(\psi_j \hat {\varphi})^{\vee}\|_{L^p}, &\quad q=\infty.
	\end{cases}
 \end{equation*}
 Note that $\|\varphi\|_{\dot{H}^{s,p}(\mathbb R^n)}=0$ or $\|\varphi\|_{\dot{B}_{p,q}^s}=0$ if and only if $\operatorname{supp} \hat{\varphi}= \{0\}$, i.e. $u$ is a polynomial.

 The nonhomogeneous Sobolev space is defined as
 \[{H}^{s,p}(\mathbb R^n)=\{\varphi\in \mathcal{S}'(\mathbb R^n): [(1+|\xi|^2)^{\frac{s}{2}}\hat{\varphi}]^{\vee}\in L^{p}(\mathbb R^n)\},\]
 with the norm $\|\varphi\|_{{H}^{s,p}(\mathbb R^n)}:=\|[(1+|\xi|^2)^{\frac{\xi}{2}}\hat{\varphi}]^{\vee}\|_{L^{p}(\mathbb R^n)}$.
The nonhomogeneous Besov space is given by
\[{B}_{p,q}^{s}(\mathbb R^n)=\{\varphi\in \mathcal{S}'(\mathbb R^n): \|\varphi\|_{{B}_{p,q}^{s}(\mathbb R^n)}<\infty\}\]
where
\[\|\varphi\|_{{B}_{p,q}^s}:= \|(\phi\hat{\varphi})^{\vee}\|_{L^{p}(\mathbb R^n)}+
 	\begin{cases}
 		(\sum_{j=1}^{\infty}(2^{js}\|(\psi_j \hat {\varphi})^{\vee}\|_{L^p})^q)^{1/q},& \quad 1\leq q <\infty\\
		\sup_{j\geq 1} 2^{js} \|(\psi_j \hat {\varphi})^{\vee}\|_{L^p}, &\quad q=\infty.
	\end{cases}.\]	
	{{Notice that $L^p(\mathbb R^n)= \dot{B}^{0}_{2,p}(\mathbb R^n)$} and $\dot{{B}}_{p,q}^s(\mathbb R^n)\hookrightarrow \dot{B}_{q,p}^s(\mathbb R^n)$ if $q\leq p$.}

\section{The proof of the main results}
Before proving our main results, we give the following two lemmas:

\begin{lem}
For $0<\beta\leq n$,
\begin{enumerate}
	\item If $\beta\neq \frac{n}{m}$ for some positive integer m, then we have
	\[K_t(x)=\sum_{k=1}^{[\frac{n}{\beta} ]}C_k|x|^{-n+\beta k}t^{-\alpha k}+Ct^{-\frac{n}{\beta}\alpha}W({x}{t^{-\frac{\alpha}{\beta}}})\]
	where $W(x)\in L^{\infty}(\mathbb R^n)$.
	\item If $\beta=\frac{n}{m}$ for some positive integer m, we have
	\[K_t(x)=\sum_{k=1}^{m-1}C_k|x|^{-n+\beta k}t^{-\alpha k}+C_mt^{-m\alpha}W_1(xt^{-\frac{\alpha}{\beta}})+Ct^{-\frac{n}{\beta}\alpha}W({x}{t^{-\frac{\alpha}{\beta}}})\]
	where $W(x)\in L^{\infty}(\mathbb R^n)$ and $W_1(x)\sim \ln|x|$, as $|x|\rightarrow 0$.
\end{enumerate}
In particular, we have $K_t(x)\notin L^{\infty}(\mathbb R^n)$ when $0<\beta \leq n$.
\end{lem}
	\begin{proof}
		(1) To start with, we rewrite the kernel as:
		\begin{align*}
			K_t(x)&=\int_{\mathbb R^n}E_{\alpha}(-it^\alpha|\xi|^\beta)e^{ix\cdot \xi}d\xi\\
			&= t^{-\frac{n}{\beta}\alpha}\int_{\mathbb R^n}E_{\alpha}(-i|\xi|^\beta)e^{ix t^{-\frac{\alpha}{\beta}}\cdot \xi}d\xi.
		\end{align*}
	By writing $\eta=x t^{-\frac{\alpha}{\beta}}$, we have
	\begin{align*}
			K_t(x)&= t^{-\frac{n}{\beta}\alpha}\int_{\mathbb R^n}E_{\alpha}(-i|\xi|^\beta)e^{i\eta \cdot \xi}d\xi\\
			&=t^{-\frac{n}{\beta}\alpha}\int_{\mathbb R^n}E_{\alpha}(-i|\xi|^\beta)\Phi (\xi) e^{i\eta \cdot \xi}d\xi+t^{-\frac{n}{\beta}\alpha}\int_{\mathbb R^n}E_{\alpha}(-i|\xi|^\beta)(1-\Phi(\xi) )e^{i\eta \cdot \xi}d\xi,
		\end{align*}
		
	where $\Phi \in C^{\infty}(\mathbb R^n)$ with $$\Phi(s) = \begin{cases} 1 & |s| \ge 2, \\ 0 &|s| \le 1. \end{cases}$$
	
Now, we define $K_t^1(\eta):=\int_{\mathbb R^n}E_{\alpha}(-i|\xi|^\beta)\Phi(\xi) e^{i\eta \cdot \xi}d\xi$ and  $K_t^2(\eta):=\int_{\mathbb R^n}E_{\alpha}(-i|\xi|^\beta)(1-\Phi(\xi) )e^{i\eta \cdot \xi}d\xi$.
One directly see that 
\[|K_t^2(\eta)|\leq C \quad \text{for}\quad t>0, \eta\in \mathbb R^n. \]
For the first term, we use the Taylor expansion of the Mittag-Leffler function to see that
\begin{equation}\label{eq6}
	 K_t^1(\eta)=\int_{\mathbb R^n}\bigg(\sum_{k=1}^{[\frac{n}{\beta} ]}C_k|\xi|^{-\beta k}+R(\xi)\bigg)\Phi(\xi)e^{i\eta\cdot\xi}d\xi,\end{equation}
where $R(\xi))= O(|\xi|^{-\beta[\frac{n}{\beta}]})$ as $|\xi|\rightarrow \infty.$
{Using the identity }
\[\mathcal F^{-1}(|\xi|^{-\theta})=C|x|^{-n+\theta},\quad \text{for} \quad {0<\theta <n},\]
we see that
\begin{align*}	
K_t^1(\eta)&=\sum_{k=1}^{[\frac{n}{\beta} ]}C_k'|\eta|^{-n+\beta k}+\int_{\mathbb R^n}\sum_{k=1}^{[\frac{n}{\beta} ]}C_k|\xi|^{-\beta k}(\Phi(\xi)-1)e^{i\eta\cdot\xi}d\xi+\tilde{W}(\eta)\\
&=\sum_{k=1}^{[\frac{n}{\beta} ]}C_k'|\eta|^{-n+\beta k}+W(\eta),
\end{align*}
where $W(\eta)\in L^\infty(\mathbb R^n)$.

(2) Note that it is sufficient to show  that
\[W_1(\eta)=\int_{\mathbb R^n}|\xi|^{-n}\Phi(\xi)e^{i\eta\cdot \xi}d\xi\sim \ln(|\eta|) \quad \text{as}\quad |\eta| \rightarrow 0.\]
In fact, we have
\[W_1(\eta)=\int_{|\xi|\geq 2}|\xi|^{-n}e^{i\eta\cdot \xi}d\xi+\int_{1\leq |\xi|\leq 2}|\xi|^{-n}\Phi(\xi)e^{i\eta\cdot \xi}d\xi.\]
Note that
\[\int_{1\leq |\xi|\leq 2}|\xi|^{-n}\Phi(\xi)e^{i\eta\cdot \xi}d\xi \leq C \quad\text{for}\quad \forall \eta \in \mathbb R^n,\]
as for the first term,
\begin{align*}
	\int_{|\xi|\geq 2}|\xi|^{-n}e^{i\eta\cdot \xi}d\xi &=
	\int_2^{+\infty}r^{-1}(r|\eta|)^{\frac{2-n}{2}}J_{\frac{n-2}{2}}(r|\eta|)dr\\
	&=\int_{2|\eta|}^{+\infty}t^{-1}t^{\frac{2-n}{2}}J_{\frac{n-2}{2}}(t)dt\\
	&=\int_{2|\eta|}^{1}t^{-1}t^{\frac{2-n}{2}}J_{\frac{n-2}{2}}(t)dt+\int_{1}^{+\infty}t^{-1}t^{\frac{2-n}{2}}J_{\frac{n-2}{2}}(t)dt,
\end{align*}
The result now follows because
\[\left|\int_{1}^{+\infty}t^{-1}t^{\frac{2-n}{2}}J_{\frac{n-2}{2}}(t)dt\right|< \infty,\]
and
\[\int_{2|\eta|}^{1}t^{-1}t^{\frac{2-n}{2}}J_{\frac{n-2}{2}}(t)dt\sim \int_{2|\eta|}^{1}\frac{dt}{t}\sim \ln|\eta|.  \]
Here we used the property  of Bessel functions that
\[J_{\frac{n-2}{2}}(t)\sim \frac{1}{\Gamma(\frac{n}{2} )}\left(\frac{t}{2} \right)^{\!\!\frac{n-2}{2}}. \]
This concludes the proof of the lemma.
	
	\end{proof}
	
\begin{rem} {{When we compare the above results with the case $\alpha=1, \beta=2$,  the asymptotic expansion of equation \eqref{eq6} isn't valid. In that case, $E_\alpha(-i|\xi|^2)= e^{-i|\xi|^2}$, thus $K_t^1(\eta)\in L^{\infty}(\eta)$}}.
	
\end{rem}

Now we are ready to prove  Theorem \ref{Lebesgue estimate} .
\begin{proof} To start with, we prove second statement is valid. First we recall that,  for any $\frac{\pi}{2}\alpha <|\arg z|\leq \pi$
\[|E_\alpha(z)|\lesssim \begin{cases}
	C_R & |z|\leq R, \\
	|z|^{-1} & |z|>R.
	\end{cases}\]
	Then  for  $\beta>n$, we have
	\begin{align*}
	|K_t(x)|&=t^{-\frac{n}{\beta}\alpha}\left|\int_{\mathbb R^n}
	E_\alpha(-i|\xi|^\beta)e^{i\eta\cdot \xi}d\xi\right|\\
	&\leq t^{-\frac{n}{\beta}\alpha}\int_{\mathbb R^n}|E_\alpha(-i|\xi|^{\beta})|d\xi \\
	&\lesssim t^{-\frac{n}{\beta}\alpha}\int_{\mathbb R^n}\frac{d\xi}{1+|\xi|^\beta}\lesssim t^{-\frac{n}{\beta}\alpha}. \\
	\end{align*}
	
As for the case $0<\beta\leq n$, and  $N$ is a dyadic number, i.e. $N=2^j$

\begin{align*}
	|K_t^N(x)|&=\int_{\mathbb R^n}E_\alpha(-it^\alpha |\xi|^\beta)\psi\left(\frac{|\xi|}{N}\right) e^{ix\cdot \xi}d\xi
	\\&= N^n\int_{\mathbb R^n}E_\alpha(-it^\alpha N^\beta |\xi|^\beta)\psi(\xi)e^{iNx\cdot \xi}d\xi\\
	&= N^n \int_0^\infty E_\alpha(-it^\alpha N^\beta |r|^\beta)\psi(r)r^{n-1}J_{\frac{n-2}{2}}(rN|x|)(N|x|)^{\frac{2-n}{2}}dr,
\end{align*}	
since $\operatorname{supp} \psi \subseteq \mathcal C=\{|\xi|: \frac{1}{2}\leq |\xi|\leq 2\}$, we have
\begin{align*}
	|K_t^N(x)|&\lesssim N^n \int_{\frac{1}{2}}^2 |E_\alpha(-it^{\alpha}N^\beta r^\beta)|dr\lesssim \frac{N^n}{1+t^\alpha N^\beta},
\end{align*}
Hence, by Young's inequality for convolutions, we have proven the theorem.
\end{proof}
{{Notice that, when $0<\beta\leq n$,  according to our method, the decay rate with respect to time is $t^{-\alpha}$, which doesn't depend on the dimension.}}

As results of the theorems, we have the following two corollaries:
\begin{cor} Let $\alpha>0, \beta>n $. If $p \in [2,\infty)$ and $t>0$, then $T_t$ maps $L^{p'}(\mathbb R^n)$ continuously to $L^{p}(\mathbb R^n)$ and we have the estimates
\[\|T_t \varphi\|_{L^{p}(\mathbb R^n)}\lesssim  t^{-\frac{2n\alpha}{\beta}(\frac{1}{2}-\frac{1}{p} )}\|\varphi\|_{L^{p'}(\mathbb R^n)}, \quad \text{ for all } \quad \varphi \in L^{p'}(\mathbb R^n). \]
	
\end{cor}
\begin{proof} Using Riesz-Thorin interpolation theorem, we immediately get the result.
	
\end{proof}

\begin{cor}Let $\alpha>0, \beta>n $, for $t>0$ then we have
\[\|T_t \varphi\|_{\dot{H}^{s,p}(\mathbb R^n)}\lesssim  t^{-\frac{2n\alpha}{\beta}(\frac{1}{2}-\frac{1}{p} )}\|\varphi\|_{\dot{H}^{s,p}(\mathbb R^n)}, \quad \text{ for all } \quad \varphi \in \dot{H}^{s,p'}(\mathbb R^n), \]
and
\[\|T_t \varphi\|_{\dot{B}_{p,q}^s(\mathbb R^n)}\lesssim  t^{-\frac{2n\alpha}{\beta}(\frac{1}{2}-\frac{1}{p} )}\|\varphi\|_{\dot{B}_{p',q}^s(\mathbb R^n)}, \quad \text{ for all } \quad \varphi \in \dot{B}_{p,q}^s(\mathbb R^n), \]
where $s\in \mathbb R$ and $2\leq p\leq \infty$ and $1\leq q \leq \infty$.
And the above estimates holds for the nonhomgeneous Sobolev spaces and Besov spaces.
\end{cor}

In the above theorems, we give the upper bound for the fundamental solution of the equation \eqref{fse}, now we prove that the estimate in the Theorem \ref{Lebesgue estimate} is optimal.

\begin{proof}
	Note that we can always choose $t, N$ such that $t^\alpha N^\beta  \geq 1$, so
	$\frac{N^n}{1+t^\alpha N^\beta}\sim t^{-\alpha}N^{n-\beta} $. Therefore, it is sufficient to show that exists $t_0, N_0$ such that
	\[\sup_{x\in \mathbb R^n}|K_t(x)|\gtrsim t^{-\alpha} N^{n-\beta}  \quad \text{for} \quad t>t_0,\quad N\geq N_0. \]
	Notice that
	\[|K_t(x)|=N^n \left|\int_{\frac{1}{2}}^2 E_\alpha(-it ^\alpha N^\beta r^\beta  )\psi(r)r^{n-1}(r|x|)^{\frac{2-n}{2}}J_{\frac{n-2}{2}}(r|x|)dr \right| := N^n |I|,\]
	we write $I$ as the following,
	\begin{align*}
		I&=\int_{\frac{1}{2}}^2 (E_\alpha(-it ^\alpha N^\beta r^\beta) -it^{-\alpha}N^{-\beta}r^{-\beta} )\psi(r)r^{n-1}(r|x|)^{\frac{2-n}{2}}J_{\frac{n-2}{2}}(r|x|)dr\\ & \quad +\int_{\frac{1}{2}}^2 it^{-\alpha}N^{-\beta}r^{-\beta} \psi(r)r^{n-1}(r|x|)^{\frac{2-n}{2}}J_{\frac{n-2}{2}}(r|x|)dr\\
		:\hspace{-0.25em}&= I_1+I_2,
	\end{align*}
	then we have $|K_t(x)|\gtrsim N^n (|I_1|-|I_2|)$.

	It is easy to check that
	\[|I_1|\lesssim (t^\alpha N^\beta)^{-2}\sim t^{- \alpha }N^{-\beta}(t^\alpha N^\beta )^{-1}.\]
For the second term, notice that
	\[|I_2|\sim N^{-\beta}t^{-\alpha}\left|\int_{\frac{1}{2}}^2 \psi(r)r^{n-1-\beta}(rN|x|)^{\frac{2-n}{2}}J_{\frac{n-2}{2}}(rN|x|)dr\right|,\]
	and we can always choose $x$ such that $2N|x|$ is small enough such that $J_{\frac{n-2}{2}}(rN|x|)>0$ is valid for all $\frac{1}{2}<r<2$.
	Therefore, we have proven that $|K_t(x)|\gtrsim N^{n-\beta}t^{-\alpha}. $
\end{proof}

 Now we give the proof for the Theorem \ref{Besov estimate}.
\begin{proof}
	(1) To start with, we define  $P_{\sim N}=\sum_{\frac{N}{2}\leq 2^j\leq 2N}P_{2^j},$ notice that $\psi\left(\frac{\xi}{2N}\right)+\psi\left(\frac{\xi}{N}\right)+\psi\left(\frac{2\xi}{N}\right)= \phi(\frac{\xi}{2^{j+1}})-\phi(\frac{\xi}{2^{j-2}})$, which equals $1$ on the support of $\psi$,
 then we have
\begin{align*}
 P_N(T_t\varphi)(x)&=\int_{\mathbb{R}^n} e^{ix\xi} E_\alpha(-it^\alpha |\xi|^\beta) \psi\left(\frac{\xi}{N}\right) \hat{\varphi}(\xi) d\xi\\
 &=\int_{\mathbb{R}^n} e^{ix\xi} E_\alpha(-it^\alpha |\xi|^\beta) \psi\left(\frac{\xi}{N}\right)\left( \psi\left(\frac{\xi}{2N}\right)+\psi\left(\frac{\xi}{N}\right)+\psi\left(\frac{2\xi}{N}\right)\right) \hat{\varphi}(\xi) d\xi
   \end{align*}
By H\"older inequality and the fact that
	$|E_\alpha(z) |\leq C $ for $ \frac{\pi}{2}\alpha<|\arg z|\leq \pi,$
we obtain that
\begin{align*}
 |P_N(T_t\varphi)(x)|  &  \le \| \widehat{ P_{\sim N} \varphi } \|_{L^r} \left\| E_\alpha(-it^\alpha |\xi|^\beta) \psi\left(\frac{\xi}{N}\right) \right\|_{L^{r'}} \\
   & \lesssim  N^{\frac{n}{r'}} \left\| \widehat{ P_{\sim N} \varphi } \right\|_{L^r}.
\end{align*}
By the Hausdorff-Young inequality, we have for $t>0$,
\begin{equation}\label{eq1}
 \| P_N(T_t\varphi)  \|_{L^\infty} \lesssim N^{\frac{n}{r'}} \|  P_{\sim N} \varphi  \|_{L^{r'}}. \end{equation}	
 In particular, we have 
 \begin{equation}\label{eq2}
 \| P_N(T_t\varphi)  \|_{L^\infty} \lesssim N^{\frac{n}{2}} \|  P_{\sim N} \varphi  \|_{L^{2}}. \end{equation}
 
Taking the sum in $N$ in inequality of  \eqref{eq1}, we get 
\begin{equation}\label{eq11}
	\| T_t\varphi  \|_{L^\infty} \lesssim \|\varphi\|_{\dot{B}_{r',1}^{\frac{n}{r'}}}.\end{equation} 	

Notice that $P_N(T_t\varphi)= K_N \ast P_{\sim N}\varphi$, and according to Theorem 1.2, we have
\begin{equation} \label{eq3}
\| P_N(T_t\varphi)  \|_{L^\infty} \lesssim t^{-\alpha} N^{n-\beta} \|  P_{\sim N} \varphi  \|_{L^{1}}.
 \end{equation}
Thus, by interpolation for $2\le r\le \infty$ between \eqref{eq1} and \eqref{eq3}, we have for $t>0$,
 \begin{equation} \label{eq10}
 \| T_t\varphi  \|_{L^\infty} \lesssim t^{-\alpha(1-\frac{2}{r})} N^{\frac{n}{r'}-\beta (1-\frac{2}{r})}\|  P_{\sim N} \varphi  \|_{L^{r'}},\end{equation}
Taking sum in $N$, we have the following result,
 
\begin{equation} \| T_t\varphi  \|_{L^\infty} \lesssim t^{-\alpha(1-\frac{2}{r})}  \|  \varphi  \|_{\dot{B}_{r',1}^{\frac{n}{r'}-\beta (1-\frac{2}{r} )}(\mathbb R^n)}. \end{equation}
 
Combine \eqref{eq11} and \eqref{eq10}, we get the inequality \eqref{eq7}.  

On the other hand, we have
\begin{equation} \label{eq4} \| P_N(T_t\varphi)  \|_{L^\infty} \lesssim N^{n} \|  P_{\sim N} \varphi  \|_{L^{1}}, \end{equation}	
  and the $L^2$ estimates:
\begin{equation}\label{eq5} \| P_N(T_t\varphi)  \|_{L^2} \lesssim  \|  P_{\sim N} \varphi  \|_{L^{2}}. \end{equation}
  
 Combining  the inequalities of \eqref{eq4} and \eqref{eq5} , we have 
 \begin{equation} \label{eq6}	
 \| P_N(T_t\varphi)  \|_{L^r} \lesssim N^{n(1-\frac{2}{r})}  \|  P_{\sim N} \varphi  \|_{L^{r'}},\end{equation}
 and summing in $N$, we have 
 \begin{equation}
 	\|T_t\varphi \|_{L^r(\mathbb R^n)}\lesssim \|\varphi  \|_{\dot{B}^{n(1-\frac{2}{r})}_{r', 2}  }
 \end{equation}

And, by interpolating between \eqref{eq3} and \eqref{eq6}, we have 
\begin{equation} \| P_N(T_t\varphi)  \|_{L^r} \lesssim t^{-\alpha(1-\frac{2}{r})} N^{(n-\beta)(1-\frac{2}{r})}  \|  P_{\sim N} \varphi  \|_{L^{r'}}.  \end{equation}
Hence, summation in $N$ gives the result of \eqref{eq8}.

{Similarly, if we combine \eqref{eq6} and \eqref{eq5}, we proved our last inequality of \eqref{eq9}.}
\end{proof}

%

$^1$ Institute of Applied Physics and Computational Mathematics, 100094, P.R.China.

    E-mail: 245440714@qq.com\\

$^2$ Department of Mathematics, Sichuan University, 610064, P.R.China.

  E-mail: zhaoshiliang@scu.edu.cn\\

$^3$ Department of Mathematics, Sichuan University, 610064, P.R.China.

 E-mail: mli@scu.edu.cn

\end{document}